\documentclass[12pt,a4paper]{amsart}

\usepackage{amscd,amssymb,amsthm}
\usepackage{jipam}

\title[Existence for Thermistor Problems on Time Scales]{Existence
of Positive Solutions for Non Local $p$-Laplacian Thermistor
Problems on Time Scales}

\author{Moulay Rchid Sidi Ammi}
\address{Department of Mathematics\\
         University of Aveiro\\
         3810-193 Aveiro, Portugal}
\email{sidiammi@mat.ua.pt}

\author{Delfim F. M. Torres}
\address{Department of Mathematics\\
         University of Aveiro\\
         3810-193 Aveiro, Portugal}
\email{delfim@ua.pt}
\urladdr{http://www.mat.ua.pt/delfim}


\begin{document}

\begin{abstract}
We make use of the Guo-Krasnoselskii fixed point theorem on cones
to prove existence of positive solutions to a non local
$p$-Laplacian boundary value problem on time scales arising in
many applications.
\end{abstract}

\maketitle


\section{Introduction}

The purpose of this paper is to prove the existence of positive
solutions for the following non local $p$-Laplacian dynamic
equation on a time scale $\mathbb{T}$:
\begin{equation} \label{eq1}
 -\left ( \phi_{p}(u^{\triangle}(t))\right)^{\nabla}=
 \frac{\lambda f(u(t))}{( \int_{0}^{T} f(u(\tau ))\, \nabla \tau )^{k}},
 \quad  \forall t \in  (0,T)_{\mathbb{T}}= \mathbb{T} \, ,
\end{equation}
subject to the boundary conditions
\begin{equation} \label{eq2}
\begin{gathered}
\phi_{p}(u^{\triangle}(0)) - \beta
\phi_{p}(u^{\triangle}(\eta))=0, \quad 0< \eta < T, \\
u(T) -\beta u(\eta)=0,
\end{gathered}
\end{equation}
where $\phi_{p}(\cdot)$ is the $p$-Laplacian operator defined by
$\phi_{p}(s)= |s|^{p-2} s$, $p>1$, $(\phi_{p})^{-1}= \phi_{q}$
with $q$ the Holder conjugate of $p$, \textrm{i.e.} $\frac{1}{p}+
\frac{1}{q}= 1$. Function
\begin{itemize}
\item[(H1)] $f: (0, T)_{\mathbb{T}} \rightarrow \mathbb{R}^{+*}$
is assumed to be continuous
\end{itemize}
($\mathbb{R}^{+*}$ denotes the positive real numbers); $\lambda$
is a dimensionless parameter that can be identified with the
square of the applied potential difference at the ends of a
conductor; $f(u)$ is the temperature dependent resistivity of the
conductor; $\beta$ is a transfer coefficient supposed to verify
$0< \beta < 1$. Different values for $p$ and $k$ are connected
with a variety of applications for both $\mathbb{T} = \mathbb{R}$
and $\mathbb{T} = \mathbb{Z}$. When $k>1$, equation \eqref{eq1}
represents the thermo-electric flow in a conductor \cite{l1}. In
the particular case $p=k=2$, \eqref{eq1} has been used to describe
the operation of thermistors, fuse wires, electric arcs and
fluorescent lights \cite{s1,s2,l2,l3}. For $k=1$, equation
\eqref{eq1} models the phenomena associated with the occurrence of
shear bands (i) in metals being deformed under high strain rates
\cite{bl,bt}, (ii) in the theory of gravitational equilibrium of
polytropic stars \cite{kn}, (iii) in the investigation of the
fully turbulent behavior of real flows, using invariant measures
for the Euler equation \cite{clmp}, (iv) in modelling aggregation
of cells via interaction with a chemical substance (chemotaxis)
\cite{w}.

The theory of dynamic equations on time scales (or, more
generally, measure chains) was introduced in 1988 by Stefan Hilger
in his PhD thesis (see \cite{h1,h2}). The theory presents a
structure where, once a result is established for a general time
scale, then special cases include a result for differential
equations (obtained by taking the time scale to be the real
numbers) and a result for difference equations (obtained by taking
the time scale to be the integers). A great deal of work has been
done since 1988, unifying and extending the theories of
differential and difference equations, and many results are now
available in the general setting of time scales -- see
\cite{a1,abra,a2,a3,b1,b2} and references therein. We point out,
however, that results concerning $p$-Laplacian problems on time
scales are scarce \cite{da}. In this paper we prove existence of
positive solutions to the problem \eqref{eq1}-\eqref{eq2} on a
general time scale $\mathbb{T}$.


\section{Preliminaries}

Our main tool to prove existence of positive solutions
(Theorem~\ref{thm25}) is the Guo-Krasnoselskii fixed point theorem
on cones.

\begin{theorem}[Guo-Krasnoselskii fixed point theorem on cones \cite{g1,k1}]
\label{thm11} Let $X$ be a Banach space and $K\subset E$ be a cone
in $X$. Assume that $\Omega _1$ and $\Omega _2$ are bounded open
subsets of $K$ with $0\in \Omega _1\subset
\overline{\Omega}_1\subset \Omega _2$ and that $G:K\to K$ is a
completely continuous operator such that
\begin{enumerate}
\item[(i)] either $\|Gw\|\leq \|w\|$, $w\in \partial \Omega _1$,
and $\|Gw\|\geq \|w\|$, $w\in \partial \Omega _2$; or
\item [(ii)]  $\|Gw\|\geq \|w\|$, $w\in \partial \Omega _1$,
and $\|Gw\|\leq \|w\|$, $w\in \partial \Omega _2$.
\end{enumerate}
Then, $G$ has a fixed point in $\overline{\Omega }_2\backslash
\Omega _1$.
\end{theorem}
Using properties of $f$ on a bounded set $(0, T)_{\mathbb{T}}$, we
construct an operator (an integral equation) whose fixed points
are solutions to the problem \eqref{eq1}-\eqref{eq2}.

Now we introduce some basic concepts of time scales that are
needed in the sequel. For deeper details the reader can see, for
instance, \cite{a1,a4,b1}. A time scale $\mathbb{T}$ is an
arbitrary nonempty closed subset of $\mathbb{R}$. The
\emph{forward jump} operator $\sigma$ and the \emph{backward jump}
operator $\rho$, both from $\mathbb{T}$ to $\mathbb{T}$, are
defined in \cite{h1}:
$$
\sigma(t)=\inf\{\tau\in\mathbb{T}:\tau> t\}\in\mathbb{T}, \quad
\rho(t)=\sup\{\tau\in\mathbb{T}:\tau< t\}\in\mathbb{T} \, .
$$
A point $t\in\mathbb{T}$ is left-dense, left-scattered,
right-dense, or right-scattered if $\rho (t)=t,\ \rho(t)<t$,
$\sigma(t)=t$, or $\sigma(t)>t$, respectively. If $\mathbb{T}$ has
a right scattered minimum $m$, define
$\mathbb{T}_{k}=\mathbb{T}-\{m\}$; otherwise set
$\mathbb{T}_{k}=\mathbb{T}$. If $\mathbb{T}$ has a left scattered
maximum $M$, define $\mathbb{T}^{k}=\mathbb{T}-\{M\}$; otherwise
set $\mathbb{T}^{k}=\mathbb{T}$.

Let $f:\mathbb{T} \to \mathbb{R}$ and $t\in \mathbb{T}^{k}$
(assume $t$ is not left-scattered if $t=\sup\mathbb{T}$), then the
delta derivative of $f$ at the point $t$ is defined to be the
number $f^{\Delta}(t)$ (provided it exists) with the property that
for each $\epsilon>0$ there is a neighborhood $U$ of $t$ such that
$$
|f(\sigma(t))-f(s)-f^{\Delta}(t)(\sigma(t)-s) |\le | \sigma(t)-s|,
\quad \mbox{for all } s\in U \, .
$$
Similarly, for $t\in \mathbb{T}$ (assume $t$ is not
right-scattered if $t=\inf\mathbb{T}$),  the nabla derivative of
$f$ at the point $t$ is defined to be the number $f^{\nabla}(t)$
(provided it exists) with the property that for each $\epsilon >0$
there is a neighborhood $U$ of $t$ such that
$$
|f(\rho(t))-f(s)-f^{\nabla}(t)(\rho(t)-s) |\le | \rho(t)-s |,
\quad \mbox{for all } s\in U \, .
$$
If $\mathbb{T}=\mathbb{R}$, then $x^\Delta(t)=x^\nabla(t)=x'(t)$.
If $\mathbb{T}=\mathbb{Z}$, then $x^\Delta(t)=x(t+1)-x(t)$ is the
forward difference operator while $x^\nabla(t)=x(t)-x(t-1)$ is the
backward difference operator.

A function $f$ is left-dense continuous ($ld$-continuous) if $f$
is continuous at each left-dense point in $\mathbb{T}$ and its
right-sided limit exists at each right-dense point in
$\mathbb{T}$. Let $f$ be $ld$-continuous. If $F^{\nabla}(t)=f(t)$,
then the nabla integral is defined by
$$
\int^b_a f(t)\nabla t=F(b)-F(a) \, ;
$$
if $F^{\Delta}(t)=f(t)$, then the delta integral is defined by
$$
\int^b_a f(t)\Delta t=F(b)-F(a) \, .
$$
In the rest of this article $\mathbb{T}$ is a closed subset of
$\mathbb{R}$ with $0\in\mathbb{T}_k$, $T\in\mathbb{T}^k$; $E=
\mathbb{C}_{ld}([0, T], \mathbb{R})$, which is a Banach space with
the maximum norm $\|u\|= \max_{[0, T]_{\mathbb{T}}}|u(t)|$.


\section{Main Results}

By a positive solution of \eqref{eq1}-\eqref{eq2} we understand a
function $u(t)$ which is positive on $(0, T)_{\mathbb{T}}$ and
satisfies \eqref{eq1} and \eqref{eq2}.

\begin{lemma}
\label{lm1} Assume that hypothesis $(H1)$ is satisfied. Then,
$u(t)$ is a solution of \eqref{eq1}-\eqref{eq2} if and only if
$u(t) \in E$ is solution of the integral equation
$$
u(t)= -\int_{0}^{t}\phi_{q}\left (g(s) \right) \triangle s + B,
$$
where
\begin{equation*}
\begin{gathered}
g(s)= \int_{0}^{s}\lambda h(u(r)) \nabla r -A,\\
 A= \phi_{p}(u^{\triangle}(0))= -\frac{\lambda \beta}{1-\beta}\int_{0}^{\eta}
h(u(r))
\nabla r, \\
h(u(t))= \frac{\lambda f(u(t))}{( \int_{0}^{T} f(u(\tau ))\, \nabla
\tau )^{k}},\\
B=u(0)= \frac{1}{1-\beta} \left \{
\int_{0}^{T}\phi_{q}(g(s))\triangle s - \beta \int_{0}^{\eta}
\phi_{q}(g(s)) \triangle s \right  \}.
\end{gathered}
\end{equation*}
\end{lemma}

\begin{proof}
We begin by proving necessity. Integrating the equation
\eqref{eq1} we have
$$
\phi_{p}(u^{\triangle}(s))= \phi_{p}(u^{\triangle}(0)) -
\int_{0}^{s}\lambda h(u(r)) \nabla r.
$$
On the other hand, by the boundary condition \eqref{eq2}
$$
\phi_{p}(u^{\triangle}(0)) = \beta \phi_{p}(u^{\triangle}(\eta))=
\beta \left ( \phi_{p}(u^{\triangle}(0)) - \int_{0}^{\eta}\lambda
h(u(r)) \nabla r \right).
$$
Then,
$$
A= \phi_{p}(u^{\triangle}(0)) = \frac{-\lambda \beta}{1- \beta}
\int_{0}^{\eta}h(u(r)) \nabla r.
$$
It follows that
\begin{equation*}
u^{\triangle}(s)=  \phi_{q}\left(-\lambda \int_{0}^{s} h(u(r))
\nabla r +A\right)=-\phi_{q}(g(s)).
\end{equation*}
Integrating the last equation we obtain
\begin{equation} \label{equa3}
u(t)= u(0) - \int_{0}^{t} \phi_{q}(g(s)) \triangle s.
\end{equation}
Moreover, by \eqref{equa3} and the boundary condition \eqref{eq2},
we have
\begin{equation*}
\begin{split}
u(0)&= u(T) + \int_{0}^{T} \phi_{q}(g(s)) \triangle s\\
&= \beta u(\eta) + \int_{0}^{T} \phi_{q}(g(s)) \triangle s \\
&= \beta \left( u(0)- \int_{0}^{\eta} \phi_{q}(g(s)) \triangle s
\right)+ \int_{0}^{T} \phi_{q}(g(s)) \triangle s.
\end{split}
\end{equation*}
Then,
$$
u(0)=B= \frac{1}{1-\beta} \left( -\beta \int_{0}^{\eta}
\phi_{q}(g(s)) \triangle s + \int_{0}^{T} \phi_{q}(g(s)) \triangle s
\right).
$$
Sufficiency follows by a simple calculation, taking the delta
derivative of $u(t)$.
\end{proof}

\begin{lemma}
Suppose $(H1)$ holds. Then, a solution $u$ of
\eqref{eq1}-\eqref{eq2} satisfies $u(t) \geq 0$ for all $t \in (0,
T)_{\mathbb{T}}$.
\end{lemma}

\begin{proof}
We have $A=\frac{-\lambda \beta}{1- \beta} \int_{0}^{\eta}h(u(r))
\nabla r \leq 0 $. Then, $g(s)= \lambda \int_{0}^{s} h(u(r))-A \geq
0$. It follows that $\phi_{p}(g(s)) \geq 0$. Since $0 < \beta < 1$,
we also have
\begin{equation*}
\begin{split}
u(0)&= B = \frac{1}{1-\beta} \left \{ \int_{0}^{T}\phi_{q}(g(s))\triangle s
- \beta \int_{0}^{\eta} \phi_{q}(g(s)) \triangle s \right \}\\
& \geq \frac{1}{1-\beta} \left \{ \beta
\int_{0}^{T}\phi_{q}(g(s))\triangle s - \beta \int_{0}^{\eta}
\phi_{q}(g(s))\triangle s  \right \} \\
&\geq 0
\end{split}
\end{equation*}
and
\begin{equation*}
\begin{split}
u(T)&= u(0)-  \int_{0}^{T}\phi_{q}(g(s))\triangle s\\
&= \frac{-\beta}{1-\beta} \int_{0}^{\eta}\phi_{q}(g(s))\triangle s +
\frac{1}{1-\beta}\int_{0}^{T}\phi_{q}(g(s))\triangle s -
\int_{0}^{T}\phi_{q}(g(s))\triangle s\\
 & = \frac{-\beta}{1-\beta}
\int_{0}^{\eta}\phi_{q}(g(s))\triangle s + \frac{\beta}{1-\beta}
\int_{0}^{T} \phi_{q}(g(s))  \triangle s\\
& = \frac{\beta}{1-\beta}   \left \{
\int_{0}^{T}\phi_{q}(g(s))\triangle s-
\int_{0}^{\eta}\phi_{q}(g(s))\triangle s \right \} \\
&\geq 0.
\end{split}
\end{equation*}
If $t \in (0, T)_{\mathbb{T}}$,
\begin{equation*}
\begin{split}
u(t)&= u(0)-  \int_{0}^{t}\phi_{q}(g(s))\triangle s\\
& \geq -  \int_{0}^{T}\phi_{q}(g(s))\triangle s + u(0)= u(T) \\
& \geq 0 \, .
\end{split}
\end{equation*}
\end{proof}

\begin{lemma}
\label{lm23} If $(H1)$ holds, then $u(T) \geq \rho u(0)$, where
$\rho = \beta \frac{T-\eta}{T- \beta \eta} \geq 0$.
\end{lemma}

\begin{proof}
We have $\phi_{p}(u^{\triangle}(s))= \phi_{p}(u^{\triangle}(0)) -
\int_{0}^{s}\lambda h(u(r)) \nabla r \leq 0$. Since $A=
\phi_{p}(u^{\triangle}(0)) \leq 0$, then $u^{\triangle} \leq 0$.
This means that $\|u\| = u(0)$, $\inf_{t \in (0,
T)_{\mathbb{T}}}u(t) = u(T)$. Moreover,
$\phi_{p}(u^{\triangle}(s))$ is non increasing which implies with
the monotonicity of $\phi_{p}$ that $u^{\triangle}$ is a non
increasing function on $(0, T)_{\mathbb{T}}$. It follows from the
concavity of $u(t)$ that each point on the chord between $(0,
u(0))$ and $(T, u(T))$ is below the graph of $u(t)$. We have
$$u(T) \geq u(0)+T \frac{u(T)-u(\eta)}{T-\eta}.$$
On other terms,
$$
Tu(\eta) - \eta u(T) \geq (T- \eta)u(0).
$$
Using the boundary condition \eqref{eq2}, it follows that
$$
\left(\frac{T}{\beta}- \eta\right) u(T) \geq (T- \eta)u(0).
$$
Then,
$$
u(T) \geq \beta \frac{T-\eta}{T- \beta \eta} u(0).
$$
\end{proof}

In order to apply Theorem~\ref{thm11}, we define the cone $K$ by
$$
K= \big\{u\in E, u \mbox { is concave on } (0, T)_{\mathbb{T}}
\mbox { and } \inf_{t \in (0, T)_{\mathbb{T}}} u(t) \geq \rho
\|u\|\big\} \, .
$$
It is easy to see that \eqref{eq1}-\eqref{eq2} has a solution
$u=u(t)$ if and only if $u$ is a fixed point of the operator $G: K
\rightarrow E$ defined by
\begin{equation} \label{eq3}
Gu(t)=  -\int_{0}^{t}\phi_{q}\left (g(s) \right) \triangle s + B,
\end{equation}
where $g$ and  $B$ are defined as in Lemma~\ref{lm1}.

\begin{lemma}
\label{lem:3.4}
Let $G$ be defined by \eqref{eq3}. Then,
\begin{description}
\item[$(i)$] $G(K) \subseteq K$;
\item[$(ii)$] $G: K \rightarrow K$ is completely continuous.
\end{description}
\end{lemma}

\begin{proof}
Condition $(i)$ holds from previous lemmas. We now prove
$(ii)$. Suppose that $D \subseteq K$ is a bounded set. Let
$u \in D$. We have:
\begin{equation*}
\begin{split}
\left|Gu(t)\right| &= \left|-\int_{0}^{t}\phi_{q}\left(g(s) \right)
\triangle s + B\right|\\
& =\left|-\int_{0}^{t}\phi_{q} \left( \int_{0}^{s} \frac{\lambda
f(u(r))}{( \int_{0}^{T} f(u(\tau ))\, \nabla \tau )^{k}}\nabla r-A
\right)  \triangle s + B\right|\\
 & \leq \int_{0}^{T}\phi_{q}\left ( \int_{0}^{s}\frac{\lambda \sup_{u
\in D} f(u)}{(T\inf_{u\in D})^{k}} \, \nabla r -A \right )\triangle s + |B|,
\end{split}
\end{equation*}

\begin{equation*}
\begin{split}
|A| &= \left|\frac{\lambda \beta}{1-\beta}\int_{0}^{\eta} h(u(r)) \nabla r\right| \\
& = \left|\frac{\lambda \beta}{1-\beta}\int_{0}^{\eta} \frac{
f(u(r))}{( \int_{0}^{T} f(u(\tau ))\, \nabla r )^{k}} \nabla r\right| \\
& \leq \frac{\lambda \beta}{1-\beta} \frac{ \sup_{u \in D}
f(u)}{(T\inf_{u\in D})^{k}} \, \, \eta.
\end{split}
\end{equation*}
In the same way, we have
\begin{equation*}
\begin{split}
|B| &\leq  \frac{1}{1-\beta} \int_{0}^{T} \phi_{q}(g(s))
\triangle s \\
& \leq  \frac{1}{1-\beta}\int_{0}^{T} \phi_{q}\left( \frac{\lambda
\sup_{u \in D}f(u)}{(T\inf_{u\in D})^{k}}\left(s
+ \frac{\beta}{1-\beta} \eta\right)\right) \triangle s \, .
\end{split}
\end{equation*}
It follows that
$$
|Gu(t)| \leq \int_{0}^{T} \phi_{q} \left ( \frac{\lambda \sup_{u \in
D} f(u)}{(T\inf_{u\in D})^{k}}\left(s+\frac{\beta \eta}{1-\beta}\right) \right)
\triangle s +|B|.
$$
As a consequence, we get
\begin{equation*}
\begin{split}
\| Gu  \| & \leq \frac{2- \beta}{1-\beta}\int_{0}^{T} \phi_{q} \left
( \frac{\lambda \sup_{u \in D} f(u)}{(T\inf_{u\in
D})^{k}}\left(s+\frac{\beta \eta}{1-\beta}\right)\right)\\
 & \leq  \frac{2}{1-\beta}\phi_{q} \left ( \frac{\lambda
\sup_{u \in D} f(u)}{(T\inf_{u\in D})^{k}} \right )\int_{0}^{T}
 \phi_{q} \left(s+\frac{\beta \eta}{1-\beta}\right)  \triangle s \, .
\end{split}
\end{equation*}
We conclude that $G(D)$ is bounded. Item $(ii)$ follows by a
standard application of Arzela-Ascoli and Lebesgue dominated
theorems.
\end{proof}

\begin{theorem}[Existence result on cones]
\label{thm25} Suppose that $(H1)$ holds. Assume furthermore that
there exist two positive numbers $a$ and $b$ such that
\begin{description}
\item[$(H2)$] $\max_{0\leq u \leq a}f(u) \leq \phi_{p}(aA_{1})$,
\item[$(H3)$] $\min_{0\leq u \leq b} f(u)\geq \phi_{p}(bB_{1})$,
\end{description}
where
$$
A_{1}=\frac{1- \beta}{T(2- \beta)} \phi_{p} \left ( \frac{1}{(T
\inf_{0 \leq u \leq a} f(u))^{k}}\left(T + \frac{\beta
\eta}{1-\beta}\right)\right)
$$
and
$$
B_{1}=\frac{1- \beta}{\beta(T- \eta)} \phi_{p}(\eta) \phi_{p}
\left(\frac{\lambda}{\left(T \sup_{0 \leq u \leq b} f(u)\right)^{k}}\right) \, .
$$
Then, there exists $0 <\lambda_* < 1$ such that the non local
$p$-Laplacian problem \eqref{eq1}-\eqref{eq2} has at least one
positive solution $\overline{u}$, $a \leq \overline{u} \leq b$,
for any $\lambda \in (0, \lambda_*)$.
\end{theorem}

\begin{proof}
Let $\Omega_{r}= \{ u\in K, \|u\| \leq r \}$, $\partial
\Omega_{r}= \{ u\in K, \|u\| = r  \}$. If $u \in \partial
\Omega_{a}$, then $0\leq u \leq a$, $t \in (0, T)_{\mathbb{T}}$.
This implies $f(u(t)) \leq  \max_{0\leq u \leq a} f(u) \leq
\phi_{p}(aA)$. We can write that
\begin{equation*}
\begin{split}
\| Gu  \| & \leq  \int_{0}^{T} \phi_{q}(g(s))\triangle s  +B\\
& \leq \int_{0}^{T} \phi_{q} \left (\int_{0}^{s} \frac{\lambda
f(u(r))}{( \int_{0}^{T} f(u(\tau ))\, \nabla \tau )^{k}} \nabla r
-A \right ) \triangle s + B \, ,
\end{split}
\end{equation*}
\begin{equation*}
|A| = \frac{\lambda \beta }{1-\beta } \int_{0}^{\eta} \frac{
f(u(r))}{( \int_{0}^{T} f(u(\tau ))\, \nabla \tau )^{k}} \nabla r
\leq  \frac{\lambda \beta }{1-\beta }\frac{(aA_{1})^{p-1}}{(T
\inf_{0 \leq u \leq a} f(u))^{k}} \eta \, ,
\end{equation*}
\begin{equation*}
g(s) \leq  \frac{\lambda (aA_{1})^{p-1}}{(T \inf_{0 \leq u \leq a}
f(u))^{k}} \left( T+ \frac{\beta \eta}{1- \beta} \right) \, .
\end{equation*}
Then,
\begin{equation*}
\begin{split}
\int_{0}^{T} \phi_{q}(g(s)) \triangle s & \leq   \phi_{q} \left (
\frac{\lambda (aA_{1})^{p-1}}{(T \inf_{0 \leq u \leq a} f(u))^{k}}
 \left( T+ \frac{\beta \eta}{1- \beta} \right)\right )T\\
 & = a A_{1} T \phi_{q} \left(
\frac{\lambda}{(T \inf_{0 \leq u \leq a} f(u))^{k}}\left( T+
\frac{\beta \eta}{1- \beta}\right)\right).
 \end{split}
\end{equation*}
Moreover,
\begin{equation*}
\begin{split}
B& = \frac{1}{1- \beta} \left (  \int_{0}^{T}\phi_{q} (g(s))
\triangle s-\beta  \int_{0}^{\eta}\phi_{q} (g(s)) \triangle s \right )\\
&  \leq \frac{1}{1- \beta} \left (  \int_{0}^{T}\phi_{q} (g(s))
\triangle s \right )  \\
&\leq  a A_{1} \frac{T}{1- \beta}\phi_{q} \left( \frac{\lambda}{(T
\inf_{0 \leq u \leq a} f(u))^{k}}\left( T+ \frac{\beta \eta}{1-
\beta}\right)\right).
\end{split}
\end{equation*}
For $A_{1}$ as in the statement of the theorem, it follows that
\begin{equation*}
\begin{split}
 \|Gu\|& \leq aA_{1} T  \frac{2- \beta}{1-\beta} \phi_{q} \left (
\frac{\lambda}{(T \inf_{0 \leq u \leq a} f(u))^{k}}\left( T+
\frac{\beta \eta}{1- \beta}\right)\right)\\
& \leq \phi_{q}(\lambda) a A_{1} T \frac{2- \beta}{1-\beta}
\phi_{q} \left( \frac{1}{(T \inf_{0 \leq u \leq a}
f(u))^{k}}\left( T+ \frac{\beta
\eta}{1- \beta}\right)\right)\\
& \leq \phi_{q}(\lambda_*) a A_{1} T \frac{2- \beta}{1-\beta}
\phi_{q} \left( \frac{1}{(T \inf_{0 \leq u \leq a}
f(u))^{k}}\left( T + \frac{\beta
\eta}{1- \beta}\right) \right)\\
&\leq  \phi_{q}(\lambda_*) a\\
 & \leq a= \|u\|.
\end{split}
\end{equation*}
If $u \in \partial \Omega_{b}$, we have
\begin{equation*}
\begin{split}
\|Gu\| &\geq -\int_{0}^{T}\phi_{q}\left (g(s) \right) \triangle s + B \\
&\geq -\int_{0}^{T}\phi_{q}\left (g(s) \right) \triangle s +
\frac{1}{1-\beta} \int_{0}^{T}\phi_{q}\left (g(s) \right)
\triangle s - \frac{\beta}{1-\beta} \int_{0}^{\eta}\phi_{q}\left
(g(s) \right) \triangle s\\
&\geq \frac{\beta}{1-\beta} \int_{0}^{T}\phi_{q}\left (g(s)
\right) \triangle s -\frac{\beta}{1-\beta}
\int_{0}^{\eta}\phi_{q}\left
(g(s) \right) \triangle s\\
& \geq \frac{\beta}{1-\beta} \int_{\eta}^{T}\phi_{q}\left (g(s)
\right) \triangle s.
\end{split}
\end{equation*}
Since $A \leq 0$, we have
\begin{equation*}
\begin{split}
g(s) & = \lambda \int_{0}^{s} h(u(r)) \nabla r -A \geq \lambda
\int_{0}^{s} h(u(r)) \nabla r \\
& \geq \lambda \int_{0}^{s} \frac{f(u)}{(T \sup_{0 \leq u \leq
b}f(u))^{k}}\\
& \geq \lambda \frac{(bB_{1})^{p-1}}{(T \sup_{0 \leq u \leq
b})^{k}} s.\\
\end{split}
\end{equation*}
Using the fact that $\phi_{q}$ is nondecreasing we get
\begin{equation*}
\begin{split}
\phi_{q}(g(s))&  \geq  \phi_{q} \left ( \lambda
\frac{(bB_{1})^{p-1}}{(T \sup_{0 \leq u \leq b})^{k}} s \right )\\
& \geq bB_{1} \phi_{q}\left(\frac{\lambda}{(T \sup
f(u))^{k}}\right) \phi_{q}(s).
\end{split}
\end{equation*}
Then, using the expression of $B_{1}$,
\begin{equation*}
\begin{split}
\|Gu\| & \geq \frac{\beta }{1- \beta} bB_{1} \phi_{q}
\left(\frac{\lambda}{(T \sup f(u))^{k}}\right) \int_{\eta}^{T}
\phi_{q}(s) \triangle s \\
& \geq bB_{1} \frac{\beta }{1- \beta}  \phi_{q}
\left(\frac{\lambda}{(T
\sup f(u))^{k}}\right) \phi_{q}(\eta)(T- \eta)\\
 &  \geq b= \|u\|.
\end{split}
\end{equation*}
As a consequence of Lemma~\ref{lem:3.4} and Theorem~\ref{thm11},
$G$ has a fixed point theorem $\overline{u}$ such that $a \leq
\overline{u} \leq b$.
\end{proof}


\section{An Example}

We consider a function $f$ which arises with the negative
coefficient thermistor (NTC-thermistor). For this example the
electrical resistivity decreases with the temperature.

\begin{corollary}
\label{cor} Assume $(H1)$ holds. If
$$
f_{0}= \lim_{u\rightarrow 0}\frac{f(u)}{\phi_{p}(u)}= 0 \, , \quad
f_{\infty}= \lim_{u\rightarrow
\infty}\frac{f(u)}{\phi_{p}(u)}=+\infty,
$$
or
$$
f_{0}= +\infty \, , \quad  f_{\infty}= 0 \, ,
$$
then problem \eqref{eq1}-\eqref{eq2} has at least a positive
solution.
\end{corollary}

\begin{proof}
If $f_{0}= 0$ then $\forall$ $A_{1} >0$ $\exists$ $a$ such that
$f(u) \leq (A_{1}u)^{p-1}$, $0 \leq u \leq a$. Similarly as above,
we can prove that $\|G u\| \leq \|u\|$, $\forall$ $u \in
\partial \Omega_{a}$. On the other hand,
if $f_{\infty}= +\infty$, then $\forall$ $B_{1} > 0$, $\exists$
$b> 0$ such that $f(u) \geq (B_{1}u)^{p-1}$, $u \geq b$. The same
way as in the proof of Theorem~\ref{thm25}, we have $\| G u\| \geq
\|u\|$, $\forall$ $u \in \partial \Omega_{b}$. By
Theorem~\ref{thm11} $G$ has a fixed point.
\end{proof}

For the NTC-thermistor the dependence of the resistivity with the
temperature can be expressed by
\begin{equation}
\label{eq:f:ex}
f(s)= \frac{1}{(1+s)^{k}} \, , \quad k \geq 2 \, .
\end{equation}
For $p=2$ we have
$$
f_{0}= \lim_{u\rightarrow 0}\frac{f(u)}{\phi_{p}(u)}= +\infty \, ,
\quad f_{\infty}= \lim_{u\rightarrow
\infty}\frac{f(u)}{\phi_{p}(u)}= 0 \, .
$$
It follows from Corollary~\ref{cor} that the boundary value
problem \eqref{eq1}-\eqref{eq2} with $p=2$ and $f$ as in
\eqref{eq:f:ex} has at least one positive solution.


\section*{Acknowledgements}

The authors were partially supported by the \emph{Portuguese
Foundation for Science and Technology} (FCT) through the
\emph{Centre for Research in Optimization and Control} (CEOC) of
the University of Aveiro, cofinanced by the European Community
fund FEDER/POCTI, and by the project SFRH/BPD/20934/2004.



\end{document}